\documentclass[a4paper,12pt]{amsart}

\pdfoutput=1

\usepackage[T1]{fontenc}
\usepackage{lmodern}
\usepackage{microtype}

\usepackage[a4paper,portrait]{geometry}
\usepackage{amssymb,amsmath}
\usepackage{mathtools}
\usepackage{verbatim} 
\usepackage{cite} 

\usepackage[english]{babel} 
\usepackage{csquotes} 

\usepackage{tikz}
\usetikzlibrary{patterns}

\usepackage[]{xcolor}
\usepackage[colorlinks=true, citecolor=blue, urlcolor=violet, breaklinks=true]{hyperref}
\usepackage{cleveref}

\newtheorem{Theorem}{Theorem}[section]
\newtheorem{Proposition}[Theorem]{Proposition}
\newtheorem{Lemma}[Theorem]{Lemma}

\newtheorem{Problem}[Theorem]{Problem}

\newtheorem{Question}[Theorem]{Question}
\theoremstyle{definition}
\newtheorem{Definition}[Theorem]{Definition}
\newtheorem{Example}[Theorem]{Example}

\DeclareMathOperator{\conv}{conv}
\DeclareMathOperator{\aff}{aff}

\newcommand{\ZZ}{\mathbb{Z}}
\newcommand{\RR}{\mathbb{R}}

\newcommand{\set}[1]{\{#1\}}
\newcommand{\with}{\ \vrule\ }

\begin{document} 

\title[Smooth centrally symmetric polytopes in dimension 3]{Smooth centrally symmetric polytopes in dimension 3 are IDP}

\author[Beck]{Matthias Beck}
\address{Department of Mathematics\\
         San Francisco State University\\
         San Francisco, CA 94132\\
         USA}
\email{mattbeck@sfsu.edu}

\author[Haase]{Christian Haase}
\address{Mathematik, Freie Universit\"at Berlin, 14195 Berlin, Germany}
\email{haase@math.fu-berlin.de}

\author[Higashitani]{Akihiro Higashitani}
\address{Department of Mathematics, Kyoto Sangyo University, Motoyama, Kamigamo, Kita-Ku, Kyoto, Japan, 603-8555}
\email{ahigashi@cc.kyoto-su.ac.jp}

\author[Hofscheier]{Johannes Hofscheier}
\address{Department of Mathematics and Statistics, McMaster University, 1280 Main Street West, Hamilton, Ontario L8S4K1, Canada}
\email{johannes.hofscheier@math.mcmaster.ca}

\author[Jochemko]{Katharina Jochemko}
\address{Department of Mathematics, Royal Institute of Technology, SE-100 44 Stockholm, Sweden}
\email{jochemko@kth.se}

\author[Katth\"an]{Lukas Katth\"an}
\address{School of Mathematics, University of Minnesota, Minneapolis, MN 55455, USA}
\email{katth001@umn.edu}

\author[Micha\l{}ek]{Mateusz Micha\l{}ek}
\address{Max Planck Institute for Mathematics in the Sciences, Inselstrasse 22, 04103 Leipzig, Germany}
\address{
Mathematical Institute of the Polish Academy of Sciences, {\'S}niadeckich 8, 00656 Warszawa, Poland}
\email{mateusz.michalek@mis.mpg.de}



\date{2 July 2018}

\begin{abstract}
  In 1997 Oda conjectured that every smooth lattice polytope has the integer decomposition property. We prove Oda's conjecture for centrally symmetric $3$-dimensional polytopes, by showing they are covered by lattice parallelepipeds and unimodular simplices.
\end{abstract}

\maketitle

\section{Introduction} 
A \textbf{lattice polytope} in $\RR^d$ is the convex hull of finitely many points in the integer lattice $\ZZ^d$. All polytopes in this paper will be assumed to be lattice polytopes. They appear naturally in a variety of different fields, such as combinatorics, commutative algebra, toric geometry and optimization, where their geometric and arithmetic behavior has been intensively studied in recent decades. In \cite{Oda}, Oda posed the following fundamental problem:
\begin{Problem}
  \label{prob:orig-oda}
  Given two lattice polytopes $P,Q \subseteq \RR^d$, when can every lattice point $p$ in the Minkowski sum $P + Q:=\{x+y\colon x \in P, y \in Q\}$ be written as the sum of two lattice points $p_1 \in P$ and $p_2 \in Q$, i.e., $p = p_1 + p_2$?
\end{Problem}
In general, for arbitrary lattice polytopes, not every lattice point in $P+Q$ is the sum of a lattice point in $P$ and a lattice point in $Q$, not even in the special case $P=Q$. For example, let $P$ be the convex hull of $(0,0,0),(1,0,0),(0,0,1),(1,2,1)$ and consider $2P$. Then $2P$ contains the lattice point $(1,1,1)$ but this cannot be written as the sum of any two lattice points in $P$. Of particular interest in this context are so-called \textit{IDP polytopes} -- a lattice polytope has the \textbf{Integer Decomposition Property} (or is \textbf{IDP} for short) if for every integer $n\geq 1$ and every lattice point $p\in nP\cap \mathbb{Z}^d$ there are lattice points $p_1,\ldots, p_n\in P\cap \mathbb{Z}^d$ such that $p=p_1+\cdots +p_n$.  IDP polytopes are of great interest when studying the arithmetic behavior of dilated polytopes (\textit{Ehrhart theory}) as well as in commutative algebra and toric geometry. The following basic fact will play a crucial role in this note: 
\begin{Proposition}[See, e.g., \cite{BGpaper}]
  Unimodular simplices, parallelepipeds, and zonotopes are IDP. 
\end{Proposition}
A natural notion in toric geometry is that of a smooth polytope: a lattice polytope $P$ is \textbf{smooth} if it is simple and if its primitive edge directions at every vertex form a basis of the lattice $(\aff P) \cap \mathbb{Z}^d$. In particular, every face of a smooth lattice polytope is itself smooth.

Due to its relation with projective normality of projective toric varieties, the following
specialization of \Cref{prob:orig-oda} was also asked by Oda~\cite{Oda}. It has since become known as \textit{Oda's Conjecture}.
\begin{Problem}[Oda's Conjecture]\label{prob:oda}
  Is every smooth lattice polytope IDP?
\end{Problem}

The purpose of this note is to prove the following case of Oda's conjecture.
\begin{Theorem}\label{thm:main}
  Every centrally symmetric $3$-dimensional smooth polytope is IDP.
\end{Theorem}

We have organized the paper as follows. In \Cref{sec:prelim} we recall some basic facts about smooth lattice polytopes which we will apply in the proof of \Cref{thm:main}. In \Cref{sec:prf-main-res} we provide a proof of \Cref{thm:main}. We have structured the crucial steps of the proof into subsequent subsections. Finally in \Cref{sec:summary} we conclude the paper with some open questions which might help to settle \Cref{prob:oda} for the $3$-dimensional case.

\section{Preliminaries}
\label{sec:prelim}

The following lemma is an immediate consequence of having IDP.
\begin{Lemma}[{\cite[p.~65]{BG}}]\label{lem:IDP-cover}
  Let $P,P_1,\ldots,P_m\subseteq \mathbb{R}^d$ be lattice polytopes such that $P=P_1\cup \cdots \cup P_m$. If $P_1,\ldots, P_m$ are IDP, then so is $P$.
\end{Lemma}

From the definition of a smooth lattice polytope, the following fact straightforwardly follows.
\begin{Lemma}\label{lem:para}
  Let $P \subseteq \RR^d$ be a smooth $d$-dimensional lattice polytope. Let $v$ be a vertex of $P$ and let $p_1, \ldots,p_d$ denote the primitive ray generators on the edges on $v$. Then the parallelepiped spanned by $p_1,\ldots,p_d$ from $v$ does not contain any lattice points aside from its vertices. 
\end{Lemma}

The following two lemmas are known to the experts -- we include them for the sake of completeness. We start by introducing some notation.
\begin{Definition}\label{def:special}
  Let $P$ be a polytope and $a$ a linear function. For a real number $c$, let $P_c$ be the hyperplane cut of $P$:
  \[
    P_c \coloneqq \{x\in P \with a(x)=c\} \text{.}
  \]
  We call $c$ \emph{special} if $P_c$ contains a vertex of $P$. For fixed $P$ and $a$ the set of special $c$'s is finite.
\end{Definition}

Recall that a fan $\Sigma$ is said to \emph{coarsen} another fan $\Sigma'$ if any $\sigma' \in \Sigma'$ is contained in some cone $\sigma \in \Sigma$. We refer to \cite[Section 1]{BG} for details and references on fans.

In the following lemma, we assume the notation as in Definition \ref{def:special}.
\begin{Lemma}\label{lem:Ram}
  For $c_1< c_2$ the normal fans of $P_{c_1}$ and $P_{c_2}$ coincide if the interval $[c_1,c_2]$ does not contain special values. If $c_2$ is the only special value in this interval, then the normal fan of $P_{c_2}$ coarsens that of $P_{c_1}$ (see Figure \ref{fig:Ram}).
\end{Lemma}
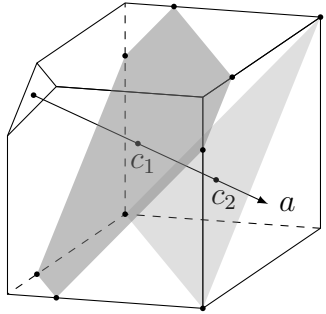
\begin{figure}[!ht]
  \begin{tikzpicture}[scale=0.7, rotate around y=-10]
    \draw (-2,-2,2) -- (-2,1,2) -- (-1,2,2) -- (2,2,2) -- (2,-2,2) -- cycle;
    \draw (-2,2,-2) -- (2,2,-2) -- (2,2,2) -- (-1,2,2) -- (-2,2,1) -- cycle;
    \draw (2,-2,-2) -- (2,2,-2) -- (2,2,2) -- (2,-2,2,2) -- cycle;
    
    
    \draw (-2,1,2) -- (-2,2,1); 
    
    \draw[-latex] (-5/3,5/3,5/3) -- (4/3,-4/3,-4/3) node[right] {$a$};
    \fill (-5/3,5/3,5/3) circle (1.5pt);
    \fill (2/3,-2/3,-2/3) circle (1.5pt) node[below,xshift=3] {$c_2$};
    \fill (-1/3,1/3,1/3) circle (1.5pt) node[below,xshift=2] {$c_1$};
    
    \draw[dashed] (2,-2,-2) -- (-2,-2,-2) -- (-2,-2,2);
    \draw[dashed] (-2,-2,-2) -- (-2,2,-2);
    
    \fill[draw=none,color=gray,opacity=0.25] (2,2,-2) -- (2,-2,2) -- (-2,-2,-2) -- cycle; 
    \fill[draw=none,color=gray,opacity=0.5] (-1,2,-2) -- (2,2,1) -- (2,1,2) -- (-1,-2,2) -- (-2,-2,1) -- (-2,1,-2) -- cycle; 
    \fill (-1,2,-2) circle (1.5pt);
    \fill (2,2,1) circle (1.5pt);
    \fill (2,1,2) circle (1.5pt);
    \fill (-1,-2,2) circle (1.5pt);
    \fill (-2,-2,1) circle (1.5pt);
    \fill (-2,1,-2) circle (1.5pt);
    \fill (2,2,-2) circle (1.5pt);
    \fill (2,-2,2) circle (1.5pt);
    \fill (-2,-2,-2) circle (1.5pt);
  \end{tikzpicture}
  \caption{Illustration of Lemma \ref{lem:Ram}.}
  \label{fig:Ram}
\end{figure}
\begin{proof}
  This is a consequence of \cite[Lemma 2.2.2]{Ram}, where we regard the hyperplane cuts $P_c$ as fibers of a projection defined by $a$, from the polytope $P$ to the line. See also \cite[Lemmas 2.4.12 and 13]{Schneider}.
\end{proof}

\begin{Lemma}\label{lem:cuts}\label{lem:sections}
  Let $P\subseteq \RR^d$ be a smooth $d$-dimensional lattice polytope, $F$ a facet of $P$ and $a \colon \RR^d \rightarrow \RR$ the primitive linear functional defining $F$, i.e., $a(\ZZ^d) = \ZZ$, $F = \set{x\in P \with a(x) = c}$ for some $c \in \ZZ$ and $a(x) \geq c$ for all $x \in P$. Then $F' \coloneqq P_{c+1}$ is a lattice polytope whose normal fan coarsens that of $F$. 
\end{Lemma}
\begin{proof}
  As $P$ is simple all but one of the edge directions from each vertex of $F$ lie in $F$. Further the smoothness condition implies that there is a lattice point on any edge adjacent to a vertex in $F$ but not contained in $F$ at lattice distance $1$ from the affine hull of $F$. Hence $F'$ is the convex hull of primitive ray generators of edges adjacent to the vertices in $F$, but not belonging to $F$.
  
  The statement about the normal fan is a general fact about simple polytopes. Let $P'\supset P$ be a (not necessarily lattice) polytope with the same normal fan as $P$. The supporting hyperplanes of $P'$ coincide with those of $P$, apart from the hyperplane supporting $F$, which is shifted parallelly by $1\gg\epsilon >0$ in the outer direction. As $P$ is simple there are no vertices of $P'$ in $P'_c$ (recall that a vertex is contained in at least $d$ facets). The values in $[c,c+1)$ are nonspecial for $P'$, as $a$ is primitive. Further, for $l\in [c,c+1]$ we have $P'_l=P_l$. By Lemma \ref{lem:Ram}, $P'_{c+1}=P_{c+1}$ may only have a fan that coarsens that of $F=P'_c$.
\end{proof}

\section{Proof of the Main Result}
\label{sec:prf-main-res}

\subsection{Covering of Lattice Polygons}
\label{ssec:cover}

\begin{figure}[ht!]
  \begin{tikzpicture}[scale=1, radius=0.1]
    \fill[pattern color=gray!30, pattern=north west lines] (1, 1) rectangle (3.5, 3.5);
    \fill[pattern color=gray!30, pattern=north west lines] (-2.5, 2.50) -- (-1, 1) -- (-2.5,1) -- cycle;
    \fill[pattern color=gray!30, pattern=north west lines] (2.5, -2.5) -- (1, -1) -- (1, -2.5) -- cycle;

    \fill[gray!40] (0,0) -- (1,0) -- (0,1) -- cycle;

    \draw[gray!70] (-2.5,-2.5) grid (3.5,3.5);
    \draw[densely dotted, very thick] (3.5, 1) -- (1, 1) -- (1, 3.5);
    \draw[densely dotted, very thick] (-2.5, 2.50) -- (-1, 1) -- (-2.5,1);
    \draw[densely dotted, very thick] (2.5, -2.5) -- (1, -1) -- (1, -2.5);

    \foreach \p in {(0,0),(1,0),(0,1)}
      \fill \p circle;

    \foreach \p in {(0,0),(1,0),(0,1)}
      \fill \p circle;

    \foreach \p in {(1,1),(-1,1),(1,-1)}
      \draw[] \p circle;

    \node[anchor=north east] at (0,0) {$v_1$};
    \node[anchor=north west] at (1,0) {$v_2$};
    \node[anchor=south east] at (0,1) {$v_3$};

    \node[anchor=south west] at (1,1) {$q_1$};
    \node[anchor=north east] at (-1,1) {$q_2$};
    \node[anchor=north east] at (1,-1) {$q_3$};

    \node at (-1.5,-1.5) {$A$};
    \node at (-0.5,2.5) {$B$};
    \node at (2.5,-0.5) {$C$};
  \end{tikzpicture}
  \caption{Illustration of the proof of \Cref{lem:sqare}.}
  \label{fig:square}
\end{figure}
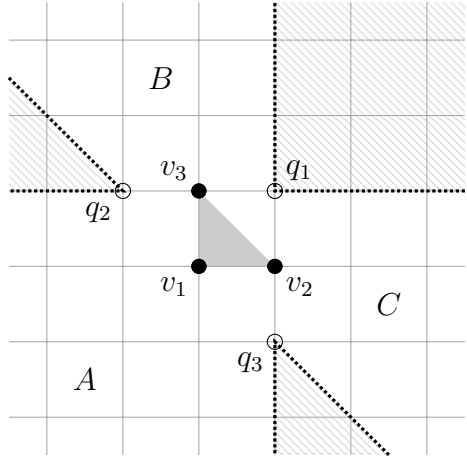

\begin{Lemma}\label{lem:sqare}
  Let $F \subseteq \RR^2$ be a smooth lattice polygon. Every unimodular simplex $\Delta \subsetneq F$ can be extended to a lattice unit square in $F$.
\end{Lemma}
\begin{proof}
  After a unimodular transformation, we may assume that $\Delta$ is the standard simplex, i.e., the central triangle in \Cref{fig:square}. Assume to the contrary that $\Delta$ cannot be extended to a unit square. This means that the three points $q_1, q_2$ and $q_3$ in \Cref{fig:square} are not contained in $F$. By convexity, it follows that $F$ does not contain any lattice point in the three shaded regions. On the other hand, we assumed that $\Delta \neq F$, so $F$ has to contain at least one further lattice point besides $v_1, v_2$ and $v_3$. Without loss of generality, we may assume that there is another lattice point in the region $A$. Further, by symmetry, we may even assume  that there is a lattice point in $A$ that is strictly to the left (and possibly below) of $v_1$ with respect to Figure \ref{fig:square}.

This implies that all further lattice points in region $B$ have to lie on the vertical line through $v_3$, as otherwise $q_2$ would lie in $F$. Let $v$ be the point furthest up on this line, where $v = v_3$ is possible. This is a vertex of $F$, and we consider the parallelepiped spanned by the two primitive ray generators on the edges on it. One of the edges goes down and leftwards into region $A$, but misses $v_1$. The other one goes down and rightwards into region $C$, possibly hitting $v_2$. Hence, $v_1$ lies in the interior of the parallelepiped, contradicting \Cref{lem:para}.
\end{proof}

\subsection{Pushing Facets}
\label{ssec:push}

\begin{Lemma}\label{lem:push}
  Let $P\subseteq \RR^3$ be a $3$-dimensional, smooth lattice polytope with a facet $F$ that is a unimodular triangle. Then (up to translation) the section of $P$ defined in \Cref{lem:sections} coincides with $rF$ for some integer $r \ge 0$.

  If $P$ has interior lattice points, (in particular, if $P=-P$) then $r\geq 2$.
\end{Lemma}
\begin{proof}
  The normal fan of $F$ has no proper coarsenings. Hence, by Lemma~\ref{lem:sections}, $F$ and $F'$ are similar and since $F$ is a unimodular triangle and $F'$ is a lattice polytope, $F'=rF$ for some integer $r\geq 0$. We note that if $r=0$ or $r=1$ then $P$ does not contain interior lattice points.
\end{proof}

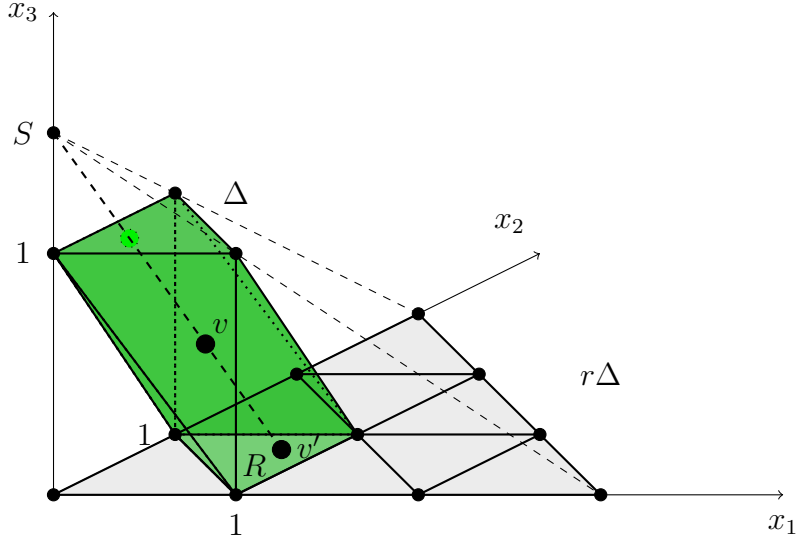
\begin{figure}[!ht]
  \begin{tikzpicture}[scale=0.8]
    \fill[green!70!black, opacity=0.5] (0,4) --(3,4) -- (2,5) -- (0,4);
    \fill[gray!30, opacity=0.5] (0,4) --(3,4) -- (2,5) -- (0,4);
    \fill[gray!30, opacity=0.5] (0,0) --(9,0) -- (6,3) -- (0,0);
    \fill[green!70!black, opacity=0.5] (0,4) --(3,0) -- (2,1) -- (0,4);
    \fill[green!70!black, opacity=0.5] (0,4) --(3,0) -- (3,4) -- (0,4);
    \fill[green!70!black, opacity=0.5] (3,0) --(5,1) -- (3,4) -- (3,0);
    \fill[green!70!black, opacity=0.5](2,1) -- (2,5) -- (0,4)--(2,1);
    \fill[green!70!black, opacity=0.5](2,1) -- (5,1) -- (2,5)--(2,1);
    \fill[green!70!black, opacity=0.5](5,1) -- (3,4)--(2,5)--(5,1);

    \filldraw [black] (0,0) circle (0.1cm);
    \filldraw [black] (0,6) circle (0.1cm); 
    \filldraw [black] (3,0) circle (0.1cm);
    \filldraw [black] (6,0) circle (0.1cm);
    \filldraw [black] (9,0) circle (0.1cm);
    \filldraw [black] (2,1) circle (0.1cm);
    \filldraw [black] (5,1) circle (0.1cm);
    \filldraw [black] (8,1) circle (0.1cm);
    \filldraw [black] (4,2) circle (0.1cm);
    \filldraw [black] (7,2) circle (0.1cm);
    \filldraw [black] (6,3) circle (0.1cm);
    \filldraw [black] (0,4) circle (0.1cm);
    \filldraw [black] (3,4) circle (0.1cm);
    \filldraw [black] (2,5) circle (0.1cm);

    \draw[->] (0,0) -- (0,8);
    \draw[->] (0,0) -- (12,0);
    \draw[->] (0,0) --  (8,4);
    \draw[thick] (0,0) --(9,0) -- (6,3) -- (0,0);
    \draw[thick] (2,1)--(8,1);
    \draw[thick] (4,2) --(7,2);
    \draw[thick] (2,1)--(3,0);
    \draw[thick] (4,2) --(6,0);
    \draw[thick] (6,0)--(8,1);
    \draw[thick] (3,0) --(7,2);

    \filldraw [dotted, fill=green!95!black] (1.25,4.25) circle (0.15cm);

    \draw[thick] (0,4) --(3,4) -- (2,5) -- (0,4);
    \draw[dashed] (0,6) --(9,0);
    \draw[dashed] (0,6) --(6,3);
    \draw[thick] (0,4) --(3,0) -- (2,1) -- (0,4);
    \draw[thick] (3,0) --(5,1) -- (3,4) -- (3,0);
    \draw[thick,dotted] (0,4) --(3,0) -- (2,1) -- (0,4);
    \draw[thick,dotted] (2,1) -- (2,5) -- (0,4)--(2,1);
    \draw[thick, dotted](2,1) -- (5,1) -- (2,5)--(2,1);

    \draw[thick, dashed](0,6) -- (3.75,0.75);

    \filldraw [black] (2.5,2.5) circle (0.15cm);
    \filldraw [black] (3.75,0.75) circle (0.15cm);

    \node at (3,-0.5) {$1$};
    \node at (-0.5,4) {$1$};
    \node at (1.5,1) {$1$};
    \node at (-0.5,6) {$S$};
    \node at (2.75,2.8) {$v$};
    \node at (4.2,0.8) {$v'$};
    \node at (12,-0.5) {$x_1$};
    \node at (-0.5,8) {$x_3$};
    \node at (7.5,4.5) {$x_2$};

    \node at (3,5) {$\Delta$};
    \node at (9,2) {$r\Delta$};
    \node at (3.3,0.5) {$R$};
  \end{tikzpicture}
  \caption{Illustration of the proof of \Cref{lem:Cayley}.}
  \label{fig:Cayley}
\end{figure}

\begin{Lemma}\label{lem:Cayley}
  Let $\Delta \subseteq \RR^2$ be a unimodular triangle and $r\geq 1$ an integer. Then the Cayley polytope of $\Delta$ and its $r$-th dilate, i.e., $Q = \conv((\Delta,1),(r\Delta,0))\subseteq \RR^3$, can be covered by unimodular simplices. In particular, it is IDP.  
\end{Lemma}
\begin{proof}
  The following straightforward argument shows that $Q$ can be covered by lattice polytopes isomorphic to either $\conv((\Delta,1), (\Delta,0))$ or $\conv((\Delta,1),(-\Delta,0))$ as illustrated by \Cref{fig:Cayley}:

  The statement is clear when $r=1$. Let $r \geq 2$. Every dilate $r\Delta$ can be triangulated by translates of $\Delta$ and $-\Delta$. Let $v$ be a point in $Q$ and let $S$ be the \emph{center of similarity} of $\Delta$ and $r\Delta$, i.e., the center of the scaling transformation which in our case is $S=(0,0,r/(r-1)) \in \RR^3$. Let $v'$ be the intersection of the straight line connecting $S$ and $v$ with the hyperplane $\{x_3=0\}$ and let $R$ be a triangle in the triangulation containing $v'$. Then $v$ is contained in $\conv((\Delta,1),(R,0))$.
  
  The polytopes $\conv((\Delta,1),(\Delta,0))$ and $\conv((\Delta,1),(-\Delta,0))$ in turn are easily seen to have a unimodular triangulation since every $3$-dimensional lattice simplex contained in $\conv((\Delta,1),(\Delta,0))$ and $\conv((\Delta,1),(-\Delta,0))$ is unimodular. One can say much more on triangulations of such polytopes e.g.~by the Cayley trick \cite{sturmfels1994newton, Santos}.
\end{proof}

\subsection{Conclusion}
\label{ssec:main}

\begin{proof}[Proof of \Cref{thm:main}]
  By \Cref{lem:IDP-cover}, it suffices to cover $P$ by parallelepipeds and unimodular simplices. Let $v\in P$ be distinct from $0$. Let $v'$ be the intersection of the half ray $\RR_{\ge0}v$ with a facet $F$ of $P$.
  \begin{enumerate}
  \item If $F$ is not a unimodular simplex, then by \Cref{lem:sqare} there exists a unit square $D$ such that $v'\in D\subseteq F$. Hence $v\in \conv(D,-D)$, which is a parallelepiped since it is unimodularly equivalent to the parallelepiped spanned by $(1,0,0), (0,1,0), (2a+1,2b+1,2\ell)$, where $\ell$ is the lattice distance of $D$ from the origin and $a, b$ are two integers. 
  \item If $F$ is a unimodular simplex let $F'$ be as in \Cref{lem:push}. If $v\in \conv(F,F')$ we are done by \Cref{lem:Cayley}. Otherwise, let $\tilde v$ be the intersection of the half ray $\RR_{\ge0}\, v$ with $F'$. We proceed as in point $(1)$ replacing $v'$ by $\tilde v$. \qedhere
  \end{enumerate} 
\end{proof}
\begin{Example}
  Let $C_d=[-1,1]^d \subset \RR^d$ and consider its $n$-th dilate $n C_d$. Then $n C_d$ is a centrally symmetric smooth polytope. By {\em chiseling off} antipodal vertices of $nC_d$ at distance $1$, there appear two unimodular facets and the smoothness is preserved. (See, e.g., \cite{CLNP} for details on chiselings.) Successive chiselings give us various examples of centrally symmetric smooth polytopes containing unimodular facets. 
\end{Example}

\section{Summary}
\label{sec:summary}

We have proved that any centrally symmetric $3$-dimensional smooth polytope $P$ is covered by parallelepipeds and unimodular simplices. It would be desirable to strengthen the statement to show that $P$ admits a unimodular covering. This would follow from a positive answer to one of the following questions.
\begin{Question}
  Do $3$-dimensional parallelepipeds admit a unimodular covering? Do centrally symmetric parallelepipeds of the form $\conv(D,-D)$ where $D$ is a unit square admit a unimodular covering?
\end{Question}

\subsection*{Acknowledgments}
The authors would like to thank the Mathematisches For\-schungs\-in\-sti\-tut Oberwolfach for hosting the Mini-Workshop \emph{Lattice polytopes: methods, advances and applications} in fall 2017 during which this project evolved. We are grateful to Joseph Gubeladze, Bernd Sturmfels, and two anonymous referees for helpful comments. Katharina Jochemko was supported by the Knut and Alice Wallenberg foundation. Lukas Katth\"an was supported by the DFG, grant KA 4128/2-1. Mateusz Micha{\l}ek was supported by the Polish National Science Centre grant no. 2015/19/D/ST1/01180. The work on this paper was completed while the fifth and sixth authors were in residence at the Mathematical Sciences Research Institute in Berkeley, California, during the Fall 2017 semester.

\bibliographystyle{plain}

\begin{thebibliography}{1}

\bibitem{BGpaper}
Winfried Bruns and Joseph Gubeladze.
\newblock Normality and covering properties of affine semigroups.
\newblock {\em Journal f\"ur die Reine und Angewandte Mathematik} 510:161--178, 1999. 

\bibitem{BG}
Winfried Bruns and Joseph Gubeladze.
\newblock Polytopes, rings, and {$K$}-theory.
\newblock {\em Springer Monographs in Mathematics}. Springer, Dordrecht, 2009.

\bibitem{CLNP}
Federico Castillo, Fu Liu, Benjamin Nill, and Andreas Paffenholz, 
\newblock {S}mooth polytopes with negative Ehrhart coefficients. 
\newblock {\em ArXiv preprint} arXiv:1704.05532, 2017.

\bibitem{Santos}
Birkett Huber, J{\"o}rg Rambau, and Francisco Santos.
\newblock {T}he {C}ayley trick, lifting subdivisions and the {B}ohne-{D}ress theorem on zonotopal tilings.
\newblock {\em Journal of the European Mathematical Society}, 2(2):179--198, 2000.

\bibitem{Oda}
Tadao Oda.
\newblock {P}roblems on {M}inkowski sums of convex lattice polytopes.
\newblock {\em ArXiv preprint} arXiv:0812.1418, 2008.

\bibitem{Ram}
J{\"o}rg Rambau. 
\newblock {P}olyhedral {S}ubdivisions and {P}rojections of {P}olytopes. 
\newblock {\em PhD thesis at Technischen Universit{\"a}t Berlin}, 1996. 

\bibitem{Schneider}
  Rolf Schneider.
  \newblock {C}onvex {B}odies: {T}he {B}runn-{M}inkowski {T}heory.
  \newblock {\em Encyclopedia of mathematics and its applications}
  {\bfseries 44}.
  \newblock Cambridge University Press, 1993.
  
\bibitem{sturmfels1994newton}
Bernd Sturmfels.
\newblock {O}n the {N}ewton polytope of the resultant.
\newblock {\em Journal of Algebraic Combinatorics}, 3(2):207--236, 1994.

\end{thebibliography}

\end{document}